\documentclass[11pt]{amsart}
\allowdisplaybreaks
\usepackage{mathrsfs}
\usepackage{cases}
\usepackage{amsmath}
\usepackage{amssymb}
\usepackage{amsfonts}
\usepackage{graphicx}
\usepackage{esint}
\usepackage{hyperref}
\makeatletter
\def\@cite#1#2{{\rm \text{#1}}}
\makeatother

\hypersetup{colorlinks=true, linkcolor=blue, citecolor=blue,}
\numberwithin{equation}{section}
\newtheorem{theorem}{Theorem}[section]

\newtheorem{corollary}[theorem]{Corollary}

\newtheorem{lemma}[theorem]{Lemma}
\newtheorem{proposition}[theorem]{Proposition}
\newtheorem{remark}[theorem]{Remark}

\newcommand{\sn}{\mathrm{sn}}
\newcommand{\vol}{\mathrm{vol}}

\begin{document}
	\title[Weighted Volume Comparison and monotonicity ]
	{Weighted  volume comparison and monotonicity for   $L^p$-bound of  Bakry-\'{E}mery Ricci  curvature}
	
	\author{Jintao Ye and Xiaohua  ${\rm ZHU}^\ddag$}
	\address{Department of Mathematics, Peking University, Beijing 100080, P. R. China}
	\email{2501110031@stu.pku.edu.cn}
	 \email{xhzhu@math.pku.edu.cn}

\thanks{$\ddag$  partially supported  by National Key R\&D Program of China
	2023YFA1009900  and 2020YFA0712800,  and NSFC 12271009.}
	
	\begin{abstract}
		We establish a Petersen-Wei type relative volume comparison theorem for weighted
		Riemannian manifolds under both $L^p$-bounds of the Bakry-\'{E}mery Ricci curvature and the gradient of potential function. As an  application, we give a modified proof for a volume comparison and monotonicity of K\"{a}hler-Ricci flow established in a recent work of  Tian-Zhang-Zhang-Zhu-Zhu with improved estimate for error term.
	\end{abstract}
	\maketitle

	\section{Introduction}

    Let $(M^n,g)$ be a complete Riemannian manifold, and $\mu(x)$ be the smallest eigenvalue of the Ricci tensor $\mathrm{Ric}(x)$
     at $x$. For $\lambda\in\mathbb{R}$ and $p>0$, we let
    \[
    \mathrm{Ric}^\lambda_-(x):=\max\{-\mu(x)+(n-1)\lambda,0\},
    \]
    and define its $L^p$ norm as follows
    \[
    \|\mathrm{Ric}^\lambda_-(x)\|_p=\left(\int_{M}\mathrm{Ric}^\lambda_-(x)^pd\mathrm{vol}\right)^{\frac{1}{p}}.
    \]
    Clearly, $\operatorname{Ric}(x)\ge (n-1)\lambda g$ if and only if $\mathrm{Ric}^\lambda_-(x)\equiv0$.

    Let $M_\lambda^n$ be an $n$-dimensional space form with constant sectional curvature $\lambda$, and $v(n,\lambda,r)$  the volume of a geodesic  $r$-ball  in $M^n_\lambda$.  In [\cite{PW}], Petersen-Wei proved the following volume comparison theorem.

    \begin{theorem}\label{PW}
    	For $\lambda\leq 0$, $p>n/2$ and $0<r\leq R$, there exists a constant $C(n,p,\lambda,R)$ such that
    	\begin{equation}\label{PW-monotone}
    	\left(\frac{\operatorname{vol}B(x,R)}{v(n,\lambda,R)}\right)^{\frac1{2p}}-\left(\frac{\operatorname{vol}B(x,r)}{v(n,\lambda,r)}\right)^{\frac1{2p}}
    	\leq C(n,p,\lambda,R)\|\mathrm{Ric}^\lambda_-\|_p^{\frac{1}{2}}.
    	\end{equation}
    	Moreover, $C(n,p,\lambda,R)=O(R^{1-\frac{n}{2p}})$ as $R\to 0$.
    \end{theorem}
    The  inequality  (\ref{PW-monotone})  generalizes the classical Bishop-Gromov's monotonicity of  Riemannian manifolds with Ricci curvature bounded below. When $r$ is sufficiently small, the volume ratio $\frac{\vol B(x,r)}{v(n,\lambda,r)}$ in (\ref{PW-monotone})  becomes almost monotone. As applications of (\ref{PW-monotone}),  the extensions of Myers'  theorem,  sphere pinching theorem  and heat kernel estimate, etc.,   have  been explored  for Riemannian manifolds with integral Ricci curvature bounds [\cite{PS, PW2}]. In [\cite{TZhangZ}], Tian-Zhang used this monotonicity property to study  the regularity of K\"{a}hler-Ricci flow on  Fano manifolds, and in particular  they proved the Hamilton-Tian conjecture   for    three-dimensional  K\"{a}hler-Ricci flow on  Fano manifolds [\cite {Ti}].

    On the other hand, there has been a  great interest in extending volume comparison results  under  Bakry-\'{E}mery (abbreviated as BE) Ricci curvature condition.  On a Riemannian manifold $(M^n,g)$,   BE Ricci curvature associated to a smooth function $f$ is a modification of Ricci curvature, defined as
    \begin{equation}\label{potential}
    	\operatorname{Ric}_f(g)=\operatorname{Ric}(g)+\operatorname{Hess}f.
    \end{equation}
    This tensor,  which was first  introduced by Bakry and \'{E}mery in the study of diffusion processes (cf. [\cite{BE,  Qi}]),   appears naturally in many contexts in differential geometry.  In particular,  the equation $\mathrm{Ric}_f(g)=\lambda g$ for some constant $\lambda$ is exactly related to  a class of canonical metrics, called   gradient Ricci solitons, which plays an important role in the singularities analysis  in  Ricci flow, see [\cite{Pe, Ha}], etc. Another natural connection between Ricci curvature and BE  Ricci curvature is from the conformal geometry, for example,  see [\cite{TZhangZ1},\cite{WZ}]. More recently, by using the conformal geometry and BE Ricci curvature, Tian-Zhang-Zhang-Zhu-Zhu   [\cite{TZZZZ}] proved a volume comparison and a  monotonicity formula   for K\"{a}hler-Ricci flow on Fano manifolds, which leads to a new essential  proof of Hamilton-Tian conjecture directly,   also  see Section \ref{sec5} below.

    Denote  $(M,g,e^{-f}d\vol_g)$ to be  a weighted Riemannian manifold.   In [\cite{WW}], Wei-Wylie established a weighted volume comparison theorem of Bishop-Gromov for weighted Riemannian manifolds  with $\operatorname{Ric}_f(g)$  bounded below, assuming the boundedness of $|\nabla f|$ or $f$. Wang-Zhu [\cite{WZ}] applied their result to study the structure of limit  spaces  of Riemannian manifolds   and Cheeger-Colding-Tian theory with BE  Ricci curvature bounded below as done in [\cite{CC, CCT}] for Riemannian manifolds with Ricci curvature  bounded below. Later, Wu also extended the Wei-Wylie's result  to the case of $L^p$-bound of  $\mathrm{Ric}_f(g)$  as in Theorem \ref{PW}  by  assuming  that $|\nabla f|$ is bounded [\cite{Wu}].   However, the pointwise  bounded  condition of  $|\nabla f|$ seems  very strong. Actually, on a complete  Riemannian manifold with gradient shrinking  Ricci soliton  $(M, g; f)$,   $f$ has a  quadratic growth  of  distance  function  of $g$ and so $|\nabla f|$ increases  as fast as the distance function (cf. [\cite{CZ, CQ}]).  

    The main purpose of the present paper is to prove  a relative volume comparison theory  for a weighted Riemannian manifold $(M^n,g,e^{-f}d\mathrm{vol}_g)$ with both  $L^p$-bounds of  $\mathrm{Ric}_f(g)$ and $|\nabla f|$.

    Before stating our main result, we introduce some notations below. For $\lambda\in\mathbb{R}$, we define a quantity as $\mathrm{Ric}^\lambda_-(x)$,
    \[
    {\mathrm{Ric}_f^\lambda}_-(x):=\max\{(n-1)\lambda-\mu_f(x),\,0\},
    \]
    where $\mu_f(x)$ is the smallest eigenvalue of $\operatorname{Ric}_f(g)$ at $x$. When $\lambda=0$, we also denote it by  ${\mathrm{Ric}_f}_-(x)$.  For $x\in M$, $p>0$, $a\geq 0$, we denote  a weighted $L^p$-norm on a geodesic  $R$-ball  $B(x,R)$
    for   a measurable function $\phi$ by
    \[
    \|\phi \|_{p,f,a}(x,R):=\left(\int_{B(x,R)}|\phi|^pe^{-a\cdot d(x, \cdot)}d\mathrm{vol}_{f}\right)^\frac{1}{p},
    \]
    where $d\mathrm{vol}_f=e^{-f}d\mathrm{vol}$. When $a=0$, we also denote it as $\|\cdot\|_{p,f}(x,R)$.

    For any smooth vector field $V$ on $(M,g)$ and $a\geq 0$, we define
    \[
    \rho_a(V)(x):= \max\{|V|_g(x)-a,0\}.
    \]
    Moreover, for $a\geq 0$, we fix a point $O\in M_\lambda^n$ and define a weighted volume of the geodesic $r$-ball $B(O,r)$ in $M_\lambda^n$  as
    \[
    v_a(n,\lambda,r):=\int_{B(O,r)}e^{a\cdot d(O,\cdot)}d\vol(\cdot).
    \]
    Clearly, $v(n,\lambda,r)\leq v_a(n,\lambda
    ,r)\leq e^{ar}v(n,\lambda,r)$.

The following is our main result in this paper.

    \begin{theorem}\label{YeZ-theorem}
    	Let $(M^n,g,e^{-f}d\operatorname{vol})$ be a weighted Riemannian manifold and $x\in M$. Then for $\lambda\leq 0$, $p>n/2$, $a\geq 0$ and $0<r<R$ we have
    	\begin{equation}\label{YeZ-monotone}
    	\begin{aligned}
    		&\left(\frac{\vol_f B(x,R)}{v_a(n,\lambda,R)}\right)^{\frac{1}{2p}}-\left(\frac{\vol_f B(x,r)}{v_a(n,\lambda,r)}\right)^{\frac{1}{2p}}\\
    		&\leq
    		C(n,p,a,\lambda,R)\left(\|\rho_a(\nabla f)\|_{2p,f,a}(x,R)+\|{\mathrm{Ric}_f^\lambda}_-\|^{\frac{1}{2}}_{p,f,a}(x,R)\right),
    	\end{aligned}
    	\end{equation}
    	where $C(n,p,a,\lambda,R)=O(R^{1-\frac{n}{2p}})$ as $R\to 0$.
    \end{theorem}
    
    When $\lambda=0$,  the inequality (\ref{YeZ-monotone}) is form-invariant under simultaneous rescalings of  metric and  parameter $a$. More precisely, for any positive  $c\le 1$, 
 (\ref{YeZ-monotone}) for  parameter $a$ and
two geodesic balls  $B(x,R)$ and  $B(x,r)$ on the rescaled weighted manifold
    \[ \frac{1}{c}M=(M^n,\frac{1}{c^2}g,\frac{e^{-f}}{c^n}d\vol_g), 
    \]
   is equivalent to one  for  parameter $a/c$  and two geodesic balls  $B(x, cR)$ and  $B(x, cr)$  on  $(M,g,e^{-f}d\vol)$.
    
    \begin{remark}
    When $|\nabla f|\leq a$ holds pointwise,   $\rho_a(\nabla f)\equiv 0$, and the terms involving $\rho_a(\nabla f)$ disappear in the right-hand side.   Then we deduce that
    \begin{align*}
    	&\left(\frac{\vol_fB(x,R)}{v_a(n,\lambda,R)}\right)^{\frac{1}{2p}}-\left(\frac{\vol_fB(x,r)}{v_a(n,\lambda,r)}\right)^{\frac{1}{2p}}\leq C(n,p,a,\lambda,R)\|{\mathrm{Ric}_f^\lambda}_-\|_{p,f,a}^{\frac{1}{2}}(x,R),
    \end{align*}
    which is analogous to Wu's result in [\cite{Wu}].
    \end{remark}
    \begin{remark}
    	When $\mathrm{Ric}_f\geq (n-1)\lambda g$ holds pointwise, i.e., ${\mathrm{Ric}_f^\lambda}_-\equiv 0$,  we deduce that
    	\begin{align*}
    		\left(\frac{\vol_f B(x,R)}{v_a(n,\lambda,R)}\right)^{\frac{1}{2p}}-\left(\frac{\vol_f B(x,r)}{v_a(n,\lambda,r)}\right)^{\frac{1}{2p}}\leq
    		C(n,p,a,\lambda,R)\|\rho_a(\nabla f)\|_{2p,f,a}(x,R),
    	\end{align*}
    	which can be viewed as an extension of Wei-Wylie's result  in the case of bounded gradient of $f$ [\cite{WW}].
    \end{remark}

   By  (\ref{YeZ-monotone}),  we also see that the volume  ratio of  $\frac{\vol_f B(x,r)}{v_a(n,\lambda,r)}$ is almost
 decreasing    as $r\to 0$ when  the $L^p$-norm of ${\mathrm{Ric}_f^\lambda}_-$ and $L^{2p}$-norm of $\nabla f$ are uniformly bounded as well as  $f$ is uniformly bounded. Then as an application of Theorem \ref{YeZ-theorem}, we give a new proof of the volume comparison for the Kähler–Ricci flow obtained in [\cite{TZZZZ}, Corollary 1.2] as follows. 
 
    \begin{corollary}\label{TZZZZ}
    	Let  $(M, g_0)$  be a Fano manifold of  complex dimension $m=n/2$  with its K\"ahler form $\omega_{g_0}\in 2\pi c_1(M)$,  and $(M, g(t, \cdot))$ ($t>0$)  a  K\"ahler-Ricci flow  with $g_0$ as  an initial metric,
    	\begin{equation}
    	\partial_tg(t,\cdot)=-{\rm Ric}(g(t, \cdot))+g(t,\cdot).  \label{krf}
    	\end{equation}
    	Then  for a given $q\in (n,+\infty)$,   there exists a constant $b_q$ depending only on $q$ and the initial metric $g_0$ such that for any  $0< r\le R \le {\rm diam}(M, g(t))$ it holds
    	\begin{equation}\label{TZZZZ-monotone}
    	\frac{|B(x,R,t)|_{g(t)}}{R^n}-\frac{|B(x,r,t)|_{g(t)}}{r^n}\leq b_qR^{1-\frac{n}{q}},
    	\end{equation}
	where  $|B(x,r,t)|_{g(t)}$ is  the volume of geodesic   $r$-ball associated to metric $g(t)$.
    \end{corollary}

  We note that  (\ref{TZZZZ-monotone}) was   proved  just for $q\in (n, 2n)$
   in [\cite{TZZZZ}, Corollary 1.2] while the present inequality holds for any $q>n$. This means that the  order of error term for the volume monotonicity can be close to the linear order of radius $R$.

    This paper is organized as follows. We first review some basic notation and derive a Riccati-type inequality in Section \ref{sec2}, then we give some types of $L^p$-estimate for the Laplacian of distance function in Section \ref{sec3}. Theorem  \ref{YeZ-theorem} and Corollary \ref{TZZZZ} will be proved in Section \ref{sec4} and Section \ref{sec5},  respectively.

$\mathbf{Acknowledgments.}$ The authors would like to thank Q. Zhang for many valuable conversations in their paper.

\section{Riccati-type inequality}\label{sec2}

Classical volume comparison theorems are often reduced to comparing the mean curvature of geodesic spheres,  i.e.  the Laplacian of distance function.
 In this section, we define a mean curvature error term $\psi$ between a weighted manifold $(M,g,e^{-f}d\vol)$ and the space form $M_\lambda^n$ with constant curvature $\lambda$,
 and  then we derive a Riccati-type inequality for it.

Let $(M^n,g)$ be a complete Riemannian manifold. Fix a point $x\in M$ and consider its segment domain
\begin{align*}
\mathrm{seg}(x):=\{&v\in T_xM\,|\,\operatorname{exp}_x(tv),\,0\leq t\leq 1\,\text{is the}\\ &\text{unique minimizing geodesic from} \,x\,\text{to}\,\operatorname{exp}_x(v)\}.
\end{align*}
On an open set  $\operatorname{exp}_x(\mathrm{seg}(x))\subset M$, we  use  polar coordinates $(r,\theta)\in \mathbb{R}^+\times S^{n-1}$ so that the volume form can be written as
\[d\mathrm{vol}=\omega(r,\theta)dr\wedge d\theta,\]
where $d\theta$ denotes the volume form on the standard unit sphere $(S^{n-1},g_{S^{n-1}})$.
We extend $\omega(r,\theta)$ by zero beyond the cut distance in each direction.

Let $r(\cdot)=d(\cdot,x)$ be a distance function on $M$  from $x$ and $h$ the Laplacian of $r$. An immediate computation in the polar coordinates shows that
\begin{align*}
	\partial_r\omega&=h\omega,\\
	\omega(r,\theta)&=O(r^{n-1}),\, r\to 0,\\
	h(r,\theta)&=\frac{n-1}{r}+O(r),\, r\to 0.
\end{align*}
In particular,  on  $(M_\lambda^n,O)$,  the Laplacian  $h_\lambda$ of distance function  satisfies
\begin{align*}
	h_\lambda(r)=(n-1)\frac{\operatorname{sn}_\lambda'(r)}{\operatorname{sn}_\lambda(r)},
\end{align*}
where $\omega_\lambda(r)=\operatorname{sn}_\lambda^{n-1}(r)$ is the volume element on  the Gauss sphere $\partial B(O,r)$, and the function $\sn_\lambda(r)$ is defined by
\begin{align}
\operatorname{sn}_\lambda(r)=
\begin{cases}
	r,&\lambda=0\\
	\operatorname{sin}(\sqrt{\lambda}r)/\sqrt{\lambda},&\lambda>0\\
	\operatorname{sinh}(\sqrt{-\lambda}r)/\sqrt{-\lambda},&\lambda<0.
\end{cases}
\end{align}

Associated to  BE Ricci curvature, it is natural to consider the following weighted Laplacian
\[
\Delta_f:=\Delta-\nabla f\cdot\nabla,
\]
see [\cite{WZ, TZZ}], etc. Then $\Delta_f r=\Delta r-\partial_rf=h-\partial_rf$,  which is  denoted  by $h_f$ for convenience.  Since the weighted volume form is given by
\[
d\operatorname{vol}_f=e^{-f}\omega(r,\theta) dr\wedge d\theta=\omega_f(r,\theta) dr\wedge d\theta,
\]
It follows that
 $$\partial_r\omega_f=h_f\omega_f.$$

Motivated  by [\cite{PW}],  we compute the mean curvature error term
\begin{equation}
	\psi(r,\theta)=(h_f-h_\lambda-a)_+(r,\theta).\label{error}
\end{equation}
By the Bochner formula, for any function $u\in C^3(M)$, we see that
	\begin{equation}
	\frac{1}{2}\Delta_f|\nabla u|^2=\langle\nabla(\Delta_fu),\nabla u\rangle+|\operatorname{Hess}u|^2+\operatorname{Ric}_f(\nabla u,\nabla u).\label{Bochner}
	\end{equation}
Applying (\ref{Bochner})
 to  the distance function $r$,  it yields
\[
\partial_rh_f+|\operatorname{Hess}r|^2=-\operatorname{Ric}_f(\partial_r,\partial_r).
\]
By choosing  a unit orthonormal frame $(e_i)$ with $e_1=\partial_r$,   it follows  from the Cauchy-Schwarz inequality that
\[
|\operatorname{Hess}r|^2=\sum_{i,j}s_{ij}^2\geq \sum_{i=1}^ns_{ii}^2=\sum_{i=2}^{n}s_{ii}^2\geq \frac{1}{n-1}(\Delta r)^2.
\]
Thus we obtain
\[
h_f'+\frac{h^2}{n-1}\leq -\operatorname{Ric}_f(\partial_r,\partial_r).
\]
Recall that in the model space,
\[
h_\lambda'+\frac{h_\lambda^2}{n-1}=-(n-1)\lambda.
\]
Subtracting this into  the previous inequality gives
\begin{equation}
\left(h_f-h_\lambda\right)'+\frac{1}{n-1}\left(h-h_\lambda\right)\left(h+h_\lambda\right)\leq  {\operatorname{Ric}_f^\lambda}_-.\label{pre-r-ineq}
\end{equation}

\begin{lemma}[Riccati-type Inequality]\label{riccati}
	For  $\lambda\leq 0$, $a\geq 0$ and $r>0$,   we have
	\begin{equation}
		\psi'+\frac{\psi^2}{n-1}\leq -\frac{2\sn_\lambda'}{\sn_\lambda}\psi+\frac{2}{n-1}\rho_a(\nabla f)\psi+\frac{2\sn_\lambda'}{\sn_\lambda}\rho_a(\nabla f)+{\operatorname{Ric}_f^\lambda}_-.\label{rineq}
	\end{equation}
\end{lemma}

\begin{proof}
	We may assume that $\psi>0$ and $\psi=h_f-h_\lambda-a$.  Substituting  it into (\ref{pre-r-ineq}) and by the fact  $h=h_f+\partial_rf$,  we have
	\[
	\psi'+\frac{1}{n-1}(\psi+\partial_rf+a)(\psi+2h_\lambda+\partial_rf+a)\leq {\operatorname{Ric}_f^\lambda}_-.
	\]
	Expanding the second  term  at the  left-hand side yields
	\begin{align*}
		\psi'+\frac{\psi^2}{n-1}&\leq -\frac{2\sn_\lambda'}{\sn_\lambda}\psi-\frac{1}{n-1}\big[(\partial_rf+a)^2+(2\psi+2h_\lambda)(\partial_rf+a)\big]+{\operatorname{Ric}_f^\lambda}_-.
	\end{align*}
Note that  $\partial_rf+a\geq -\rho_a(\nabla f)$ and $h_\lambda\geq 0$.   Thus we obtain
\begin{align*}
	\psi'+\frac{\psi^2}{n-1}&\leq -\frac{2\sn_\lambda'}{\sn_\lambda}\psi-\frac{1}{n-1}(2\psi+2h_\lambda)(\partial_rf+a)+{\operatorname{Ric}_f^\lambda}_-\\
	&\leq -\frac{2\sn_\lambda'}{\sn_\lambda}\psi+\frac{1}{n-1}(2\psi+2h_\lambda)\rho_a(\nabla f)+{\operatorname{Ric}_f^\lambda}_-,
\end{align*}
which is equivalent to the desired inequality (\ref{rineq}).
\end{proof}

Compared to the corresponding  Riccati-type inequality considered  in [\cite{WW}] or  [\cite{Wu}],  there are  two new terms $\rho_a(\nabla f)\psi$ and $\frac{\sn_\lambda'}{\sn_\lambda}\rho_a(\nabla f)$ in (\ref{rineq}), which will be handled  in next section.

\section{$L^p$-estimate  for $\psi$}\label{sec3}

In this section, we use Lemma \ref{riccati} to compute   $L^p$-estimate  for $\psi$. The key new ingredients are the use of the term $\frac{2\sn_\lambda'}{\sn_\lambda}\psi$ in (\ref{rineq}) and Young's inequality, which avoids the singular radial term and the mismatch between the powers of the weighted volume elements occurring in the first version of this paper. First we have

\begin{lemma}\label{intesti}
	For any $p>\frac{n}{2}$, $\lambda\leq 0$, $a\geq 0$ and $R>0$, there exist constants $C_1(n,p),\,C_2(n,p)$ and $C_3(n,p)$ such that
	\begin{equation}
		\begin{aligned}
			&\quad\,\int_{0}^Rr\psi^{2p}\omega_fe^{-ar}dr+\int_0^R\psi^{2p-1}\omega_fe^{-ar}dr
			\\&\leq C_1(n,p)R\int_{0}^{R}\rho_a(\nabla f)^{2p}\omega_fe^{-ar}dr\\&+C_2(n,p)R\int_{0}^R\left({\mathrm{Ric}_f^\lambda}_-\right)^{p}\omega_fe^{-ar}dr\\&+C_3(n,p)\left(\max_{0\leq t\leq R}\frac{2t\sn_\lambda'(t)}{\sn_\lambda(t)}\right)^{2p-1}\int_{0}^R\rho_a(\nabla f)^{2p-1}\omega_fe^{-ar}dr,\label{esti1}
		\end{aligned}
	\end{equation}
	where the angular coordinates $\theta$ are fixed.
\end{lemma}

\begin{proof}
	Multiplying both sides of  (\ref{rineq}) in Lemma \ref{riccati} by $r\psi^{2p-2}\omega_fe^{-ar}$, we have
	\begin{align*}
		&\frac{1}{2p-1}r(\psi^{2p-1})'\omega_fe^{-ar}+\frac{1}{n-1}r\psi^{2p}\omega_fe^{-ar}\\&\leq -\frac{2r\sn_\lambda'(r)}{\sn_\lambda(r)}\psi^{2p-1}\omega_fe^{-ar}+\frac{2}{n-1}r\psi^{2p-1}\rho_a(\nabla f)\omega_fe^{-ar}\\&+\frac{2r\sn_\lambda'(r)}{\sn_\lambda(r)}\psi^{2p-2}\rho_a(\nabla f)\omega_fe^{-ar}+r\psi^{2p-2}{\mathrm{Ric}^\lambda_f}_-\omega_fe^{-ar}.
	\end{align*}
	Observe that
	\begin{align*}
		&\quad r(\psi^{2p-1})'\omega_fe^{-ar}\\&=(r\psi^{2p-1}\omega_fe^{-ar})'-\psi^{2p-1}\omega_fe^{-ar}-r\psi^{2p-1}h_f\omega_fe^{-ar}+ar\psi^{2p-1}\omega_fe^{-ar}\\
		&\geq (r\psi^{2p-1}\omega_fe^{-ar})'-\psi^{2p-1}\omega_fe^{-ar}-r\psi^{2p}\omega_fe^{-ar}-(n-1)\frac{r\sn_\lambda'(r)}{\sn_\lambda(r)}\psi^{2p-1}\omega_fe^{-ar},
	\end{align*}
	where we used $h_f\leq \psi+h_\lambda+a$. Substituting this into the previous inequality, it yields
	\begin{equation*}
		\begin{aligned}
			&\quad\,\frac{1}{2p-1}(r\psi^{2p-1}\omega_fe^{-ar})'+\left(\frac{1}{n-1}-\frac{1}{2p-1}\right)r\psi^{2p}\omega_fe^{-ar}\\
			&+\left(2-\frac{n-1}{2p-1}\right)\frac{r\sn_\lambda'(r)}{\sn_\lambda(r)}\psi^{2p-1}\omega_fe^{-ar}-\frac{1}{2p-1}\psi^{2p-1}\omega_fe^{-ar}\\
			&\leq \frac{2}{n-1}r\psi^{2p-1}\rho_a(\nabla f)\omega_fe^{-ar}+\frac{2r\sn_\lambda'(r)}{\sn_\lambda(r)}\psi^{2p-2}\rho_a(\nabla f)\omega_fe^{-ar}\\
			&+r\psi^{2p-2}{\mathrm{Ric}^\lambda_f}_-\omega_fe^{-ar}.
		\end{aligned}
	\end{equation*}
	Since $\lambda\leq 0$, we have $\frac{r\sn_\lambda'(r)}{\sn_\lambda(r)}\geq 1$, and $2-\frac{n-1}{2p-1}>0$ by $p>\frac{n}{2}$. Thus
	\begin{equation}
		\begin{aligned}
			&\quad\,\frac{1}{2p-1}(r\psi^{2p-1}\omega_fe^{-ar})'+\frac{2p-n}{(n-1)(2p-1)}r\psi^{2p}\omega_fe^{-ar}\\
			&+\frac{4p-n-2}{2p-1}\psi^{2p-1}\omega_fe^{-ar}\\
			&\leq \frac{2}{n-1}r\psi^{2p-1}\rho_a(\nabla f)\omega_fe^{-ar}+\frac{2r\sn_\lambda'(r)}{\sn_\lambda(r)}\psi^{2p-2}\rho_a(\nabla f)\omega_fe^{-ar}\\
			&+r\psi^{2p-2}{\mathrm{Ric}^\lambda_f}_-\omega_fe^{-ar}.\label{rineq-deal}
		\end{aligned}
	\end{equation}
	
	Integrating the inequality (\ref{rineq-deal}) gives
	\begin{align*}
		&\quad\,\frac{1}{2p-1}r\psi^{2p-1}\omega_fe^{-ar}\bigg|_0^R+\frac{2p-n}{(n-1)(2p-1)}\int_{0}^{R}r\psi^{2p}\omega_fe^{-ar}dr\\
		&+\frac{4p-n-2}{2p-1}\int_0^R\psi^{2p-1}\omega_fe^{-ar}dr\\
		&\leq \frac{2}{n-1}\int_{0}^Rr\psi^{2p-1}\rho_a(\nabla f)\omega_fe^{-ar}dr+\int_{0}^Rr\psi^{2p-2}{\mathrm{Ric}^\lambda_f}_-\omega_fe^{-ar}dr\\
		&+\max_{0\leq t\leq R}\frac{2t\sn_\lambda'(t)}{\sn_\lambda(t)}\int_{0}^R\psi^{2p-2}\rho_a(\nabla f)\omega_fe^{-ar}dr.
	\end{align*}
	Since $\omega_f=O(r^{n-1})$ as $r\to 0$, we have
	$r\psi^{2p-1}\omega_fe^{-ar}\big|_0^R\geq 0$.  Note that
	$\frac{2p-n}{(n-1)(2p-1)}>0$ and $\frac{4p-n-2}{2p-1}>0$ by $p>n/2$.  Thus  we derive
	\begin{equation}
		\begin{aligned}
			&\frac{2p-n}{(n-1)(2p-1)}\int_{0}^{R}r\psi^{2p}\omega_fe^{-ar}dr+\frac{4p-n-2}{2p-1}\int_0^R\psi^{2p-1}\omega_fe^{-ar}dr\\
			&\leq \frac{2}{n-1}\int_{0}^Rr\psi^{2p-1}\rho_a(\nabla f)\omega_fe^{-ar}dr+\int_{0}^Rr\psi^{2p-2}{\mathrm{Ric}^\lambda_f}_-\omega_fe^{-ar}dr\\
			&+\max_{0\leq t\leq R}\frac{2t\sn_\lambda'(t)}{\sn_\lambda(t)}\int_{0}^R\psi^{2p-2}\rho_a(\nabla f)\omega_fe^{-ar}dr.
		\end{aligned}
		\label{rineq-int}
	\end{equation}
	On the other hand, by Young's inequality, we can see that
	\begin{align}
		r\psi^{2p-1}\rho_a(\nabla f)&\leq \epsilon\,r \psi^{2p}+C_\epsilon\,r\rho_a(\nabla f)^{2p} ,\label{young1}\\
		r\psi^{2p-2}{\mathrm{Ric}_f^\lambda}_-&\leq \epsilon\,
		r\psi^{2p}+D_\epsilon\,r({\mathrm{Ric}_f^\lambda}_-)^p ,\label{young2}\\
		\psi^{2p-2}\rho_a(\nabla f) &\leq \eta\, \psi^{2p-1}+C_\eta\,\rho_a(\nabla f)^{2p-1},\label{young3}
	\end{align}
	where $\epsilon,\,\eta >0$ can be chosen small enough, and
	\[
	C_\epsilon=\frac{(2p-1)^{2p-1}}{(2p)^{2p}\epsilon^{2p-1}},\, D_\epsilon=\frac{(p-1)^{p-1}}{p^p\epsilon^{p-1}},\, C_\eta=\frac{(2p-2)^{2p-2}}{(2p-1)^{2p-1}\eta^{2p-2}}.
	\]
	Hence, substituting (\ref{young1}), (\ref{young2}) and (\ref{young3}) into (\ref{rineq-int}), we obtain
	\begin{equation*}
		\begin{aligned}
			&\quad\,\left(\frac{2p-n}{(n-1)(2p-1)}-\frac{n+1}{n-1}\epsilon\right)\int_{0}^{R}r\psi^{2p}\omega_fe^{-ar}dr\\
			&+\left(\frac{4p-n-2}{2p-1}-\eta\,\max_{0\leq t\leq R}\frac{2t\sn_\lambda'(t)}{\sn_\lambda(t)}\right)\int_0^R\psi^{2p-1}\omega_fe^{-ar}dr\\
			&\leq \frac{2C_\epsilon}{n-1}\int_{0}^Rr\rho_a(\nabla f)^{2p}\omega_fe^{-ar}dr+D_\epsilon\int_{0}^Rr\left({\mathrm{Ric}^\lambda_f}_-\right)^p\omega_fe^{-ar}dr\\
			&+C_\eta\,\max_{0\leq t\leq R}\frac{2t\sn_\lambda'(t)}{\sn_\lambda(t)}\int_{0}^R\rho_a(\nabla f)^{2p-1}\omega_fe^{-ar}dr.
		\end{aligned}
	\end{equation*}
	Since $\epsilon$ and $\eta$ are small enough, we derive
	\begin{equation*}
		\begin{aligned}
			&\quad\,\int_{0}^{R}r\psi^{2p}\omega_fe^{-ar}dr+\int_0^R\psi^{2p-1}\omega_fe^{-ar}dr\\
			&\leq C_1(n,p)\int_{0}^Rr\rho_a(\nabla f)^{2p}\omega_fe^{-ar}dr\\
			&+C_2(n,p)\int_{0}^Rr\left({\mathrm{Ric}^\lambda_f}_-\right)^p\omega_fe^{-ar}dr\\
			&+C_3(n,p)\left(\max_{0\leq t\leq R}\frac{2t\sn_\lambda'(t)}{\sn_\lambda(t)}\right)^{2p-1}\int_{0}^R\rho_a(\nabla f)^{2p-1}\omega_fe^{-ar}dr,
		\end{aligned}
	\end{equation*}
	which proves the inequality (\ref{esti1}).
\end{proof}

By Lemma \ref{intesti}, we get the following
$L^{2p}$-estimate and $L^{2p-1}$-estimate for $\psi$.
\begin{proposition}\label{meancomp}
	Let $(M,g,e^{-f}d\mathrm{vol},x)$ be a pointed weighted Riemannian manifold. Then for $\lambda\leq 0$, $p>\frac{n}{2}$ and $a\geq 0$, there exist constants $C_1(n,p)$ and $C_2(n,p)$ such that
	\begin{equation}
		\begin{aligned}
			&\|r^{\frac{1}{2p}}\psi\|^{2p}_{2p,f,a}(x,R)+\|\psi\|^{2p-1}_{2p-1,f,a}(x,R)\\
			&\leq C_1(n,p)R\left(\|\rho_a(\nabla f)\|^{2p}_{2p,f,a}(x,R)+\|{\mathrm{Ric}_f^\lambda}_-\|^p_{p,f,a}(x,R)\right)\\ &+C_2(n,p)\left(\max_{0\leq t\leq R}\frac{2t\sn_\lambda'(t)}{\sn_\lambda(t)}\right)^{2p-1}\|\rho_a(\nabla f)\|^{2p-1}_{2p-1,f,a}(x,R).
		\end{aligned}\label{meancomp-ineq}
	\end{equation}
	In particular, for any $R>0$, one of the following
	estimates must hold: either
	\begin{align}
		\|r^{\frac{1}{2p}}\psi\|_{2p,f,a}(x,R)
		\le C_1(n,p)\,R^{\frac{1}{2p}}\left(\|\rho_a(\nabla f)\|_{2p,f,a}+\|{\mathrm{Ric}_f^\lambda}_-\|^{\frac{1}{2}}_{p,f,a}\right),\label{meancomp-ineq1}
	\end{align}
	or
	\begin{align}
		\|\psi\|_{2p-1,f,a}(x,R)
		\le C_2(n,p)\max_{0\leq t\leq R}\frac{2t\sn_\lambda'(t)}{\sn_\lambda(t)}\|\rho_a(\nabla f)\|_{2p-1,f,a}(x,R).\label{meancomp-ineq2}
	\end{align}
\end{proposition}

\begin{proof}
	Integrating (\ref{esti1}) along angular coordinates, we obtain (\ref{meancomp-ineq}). It follows from (\ref{meancomp-ineq}) that at least one of the following inequalities holds: either
	\begin{align}
		\|r^{\frac{1}{2p}}\psi\|^{2p}_{2p,f,a}(x,R)\leq C_1(n,p)R\left(\|\rho_a(\nabla f)\|^{2p}_{2p,f,a}+\|{\mathrm{Ric}_f^\lambda}_-\|^p_{p,f,a}\right) \label{alt1}
	\end{align}
	or	
	\begin{align}
		\|\psi\|^{2p-1}_{2p-1,f,a}(x,R)\leq C_2(n,p)\left(\max_{0\leq t\leq R}\frac{t\sn_\lambda'(t)}{\sn_\lambda(t)}\right)^{2p-1}\|\rho_a(\nabla f)\|^{2p-1}_{2p-1,f,a}.\label{alt2}
	\end{align}
	If (\ref{alt1}) holds, then taking the $2p$-th root and using $\left(a^{2p}+b^{2p}\right)^{1/2p}\leq a+b$ ($a,b>0$), we obtain (\ref{meancomp-ineq1}).
	If (\ref{alt2}) holds, then taking the $(2p-1)$-th root we obtain (\ref{meancomp-ineq2}).
\end{proof}

\section{Proof of Theorem \ref{YeZ-theorem} }\label{sec4}

In this section, we prove Theorem  \ref{YeZ-theorem} by the $L^{p}$-estimate of $\psi$  in Proposition \ref{meancomp}.

As in [\cite{PW, Wu}], we need to compute the ratio of volume as follows.

\begin{lemma}\label{volratio}
	Fix $x\in M$. For any $\lambda\leq 0$, $a\geq 0$ and $R>0$ we have
	\begin{equation}
	\frac{d}{dR}\left(\frac{\vol_fB(x,R)}{v_a(n,\lambda,R)}\right)\leq \frac{e^{2aR}\omega_\lambda(R)|S^{n-1}|}{v_a(n,\lambda,R)^2}\int_{B(x,R)}r\psi  e^{-ar} d\operatorname{vol}_f.\label{volratio-ineq}
	\end{equation}
\end{lemma}

\begin{proof}
	Since
	 \begin{align*}
	 \frac{d}{dR}\vol_fB(x,R)&=\frac{d}{dR}\int_{0}^{R}\int_{S^{n-1}}\omega_f(r,\theta)d\theta dr=\int_{S^{n-1}}\omega_f(R,\theta)d\theta
	\end{align*}
	and
	\begin{align*}
	 \frac{d}{dR}v_a(n,\lambda,R)&=\frac{d}{dR}\int_{0}^{R}\int_{S^{n-1}}\omega_\lambda^a(r)d\theta dr=\int_{S^{n-1}}\omega_\lambda^a(R)d\theta,
	 \end{align*}
	 where $\omega_\lambda^a(r)=e^{ar}\omega_\lambda(r)$, then we have
	\begin{align*}
		&\frac{d}{dR}\left(\frac{\vol_fB(x,R)}{v_a(n,\lambda,R)}\right)\\
		&=\frac{\int_{S^{n-1}}\omega_f(R,\theta)d\theta\int_{0}^{R}\int_{S^{n-1}}\omega^a_\lambda(r)d\theta dr-\int_{0}^{R}\int_{S^{n-1}}\omega_f(r,\theta)d\theta dr \int_{S^{n-1}}\omega^a_\lambda(R)d\theta}{v_a(n,\lambda,R)^2}\\
		&=\frac{|S^{n-1}|\omega^a_\lambda(R)}{v_a(n,\lambda,R)^2}\int_{0}^{R}\omega^a_\lambda(r)\int_{S^{n-1}}\left(\frac{\omega_f(R,\theta)}{\omega^a_\lambda(R)}-\frac{\omega_f(r,\theta)}{\omega^a_\lambda(r)}\right)d\theta dr.
	\end{align*}
By Newton-Leibniz formula and the fact $(\omega_f/\omega^a_\lambda)'=(h_f-h_\lambda-a)\cdot\omega_f/\omega^a_\lambda$, it follows that
\begin{align*}
		&\frac{d}{dR}\left(\frac{\vol_fB(x,R)}{v_a(n,\lambda,R)}\right)\\
		&=\frac{|S^{n-1}|\omega^a_\lambda(R)}{v_a(n,\lambda,R)^2}\int_{0}^{R}\omega^a_\lambda(r)\int_{S^{n-1}}\int_{r}^{R}\left(\frac{\omega_f}{\omega^a_\lambda}\right)'(s)dsd\theta dr\\
		&=\frac{|S^{n-1}|\omega^a_\lambda(R)}{v_a(n,\lambda,R)^2}\int_{0}^{R}\omega^a_\lambda(r)\int_{S^{n-1}}\int_{r}^{R}(h_f-h_\lambda-a)\frac{\omega_f(s,\theta)}{\omega^a_\lambda(s)}dsd\theta dr.
\end{align*}
By exchanging the order of integration, we get
\begin{align*}
	&\frac{d}{dR}\left(\frac{\vol_fB(x,R)}{v_a(n,\lambda,R)}\right)\\
	&=\frac{\omega_\lambda^a(R)|S^{n-1}|}{v_a(n,\lambda,R)^2}\int_{S^{n-1}}\int_{0}^R\int_{r}^{R}(h_f-h_\lambda-a)\frac{\omega_f(s,\theta)}{\omega^a_\lambda(s)}\omega^a_\lambda(r)dsdrd\theta\\
	&=\frac{\omega_\lambda^a(R)|S^{n-1}|}{v_a(n,\lambda,R)^2}\int_{S^{n-1}}\int_{0}^R(h_f-h_\lambda-a)(s)\frac{\omega_f(s,\theta)}{\omega^a_\lambda(s)}\int_{0}^{s}\omega^a_\lambda(r)drdsd\theta\\
	&\leq \frac{\omega_\lambda^a(R)|S^{n-1}|}{v_a(n,\lambda,R)^2}\int_{S^{n-1}}\int_{0}^R(h_f-h_\lambda-a)_+(s)s\omega_f(s,\theta)\frac{\int_{0}^{s}\omega^a_\lambda(r)dr}{s\omega^a_\lambda(s)}dsd\theta.
\end{align*}
Thus,
\begin{align*}
		&\quad\,\frac{d}{dR}\left(\frac{\vol_fB(x,R)}{v_a(n,\lambda,R)}\right)\\&\leq \frac{e^{aR}\omega_\lambda(R)|S^{n-1}|}{v_a(n,\lambda,R)^2}\left(\max_{0\leq s\leq R}\frac{\int_{0}^{s}\omega_\lambda^a(r)dr}{s\omega_\lambda(s)}\right)\int_{S^{n-1}}\int_{0}^Rs\psi(s,\theta) e^{-as} \omega_f(s,\theta)dsd\theta.
	\end{align*}
Note that $\lambda\leq 0$ and $a\geq 0$, then 
\[
\frac{\int_{0}^{s}\omega_\lambda^a(r)dr}{s\omega_\lambda(s)}\leq e^{aR}\frac{\int_0^s\omega_\lambda(r)dr}{s\omega_\lambda(s)}\leq e^{aR},\quad \forall \,0\leq s\leq R.
\]
Hence we obtain (\ref{volratio-ineq}) from the above immediately.
\end{proof}

Furthermore, we have

\begin{proposition}\label{reduce}
	For $p>n/2$, $\lambda\leq 0$ and $a\geq 0$, we have
	\begin{equation}
		\begin{aligned}
			\frac{d}{dR}\left(\frac{\vol_fB(x,R)}{v_a(n,\lambda,R)}\right)^\frac{1}{2p}&\leq \frac{C(n,p)Re^{2aR}\omega_\lambda(R)}{\left(\int_{0}^R\omega^a_\lambda(r)dr\right)^{1+\frac{1}{2p}}}\\
			&\times\min\bigg\{R^{-\frac{1}{2p}}\|r^{\frac{1}{2p}}\psi\|_{2p,f,a}(x,R),\\
			&\big[\vol_fB(x,R)\big]^{\frac{-1}{2p(2p-1)}}\|\psi\|_{2p-1,f,a}(x,R)\bigg\}. \label{reduce-ineq}
		\end{aligned}
	\end{equation}
\end{proposition}

\begin{proof}
	By Lemma \ref{volratio} and  H\"{o}lder's inequality,
	\begin{align*}
		\frac{d}{dR}\left(\frac{\vol_fB(x,R)}{v_a(n,\lambda,R)}\right)\leq \frac{e^{2aR}\omega_\lambda(R)|S^{n-1}|}{v_a(n,\lambda,R)^2}\|r^{\frac{1}{2p}}\psi\|_{2p,f,a}\,R^{1-\frac{1}{2p}}\big[\vol_fB(x,R)\big]^{1-\frac{1}{2p}},
	\end{align*}
	thus
	\begin{equation}
	\frac{d}{dR}\left(\frac{\vol_fB(x,R)}{v_a(n,\lambda,R)}\right)^{\frac{1}{2p}}\leq \frac{C_1(n,p)Re^{2aR}\omega_\lambda(R)}{v_a(n,\lambda,R)^{1+\frac{1}{2p}}}R^{\frac{-1}{2p}}\|r^{\frac{1}{2p}}\psi\|_{2p,f,a}.\label{reduce-ineq1}
	\end{equation}
	On the other hand, we also have
	\[
	\frac{d}{dR}\left(\frac{\vol_fB(x,R)}{v_a(n,\lambda,R)}\right)\leq \frac{e^{2aR}\omega_\lambda(R)|S^{n-1}|}{v_a(n,\lambda,R)^2}\|\psi\|_{2p-1,f,a}R\,\big[\vol_fB(x,R)\big]^{1-\frac{1}{2p-1}},
	\]
	which follows that
	\begin{equation}
		\begin{aligned}
			&\frac{d}{dR}\left(\frac{\vol_fB(x,R)}{v_a(n,\lambda,R)}\right)^{\frac{1}{2p}}\\
			&\leq \frac{C_2(n,p)Re^{2aR}\omega_\lambda(R)}{v_a(n,\lambda,R)^{1+\frac{1}{2p}}}\|\psi\|_{2p-1,f,a}\big[\vol_fB(x,R)\big]^{\frac{-1}{2p(2p-1)}}.\label{reduce-ineq2}
		\end{aligned}
	\end{equation}
	Combine (\ref{reduce-ineq1}) and (\ref{reduce-ineq2}), we obtain (\ref{reduce-ineq}).
\end{proof}

Now we can finish the proof of  Theorem  \ref{YeZ-theorem}.

\begin{proof}[Proof of Theorem  \ref{YeZ-theorem}]
	From Proposition \ref{meancomp}, we know that either
	\[
	R^{\frac{-1}{2p}}\|r^{\frac{1}{2p}}\psi\|_{2p,f,a}(x,R)
	\le C_1(n,p)\,\left(\|\rho_a(\nabla f)\|_{2p,f,a}+\|{\mathrm{Ric}_f^\lambda}_-\|^{\frac{1}{2}}_{p,f,a}\right) 
	\]
	or
	\begin{align*}
	&\big[\vol_fB(x,R)\big]^{\frac{-1}{2p(2p-1)}}\|\psi\|_{2p-1,f,a}(x,R)\\
	&\le C_2(n,p)\max_{0\leq t\leq R}\frac{2t\sn_\lambda'(t)}{\sn_\lambda(t)}\|\rho_a(\nabla f)\|_{2p-1,f,a}(x,R)\big[\vol_fB(x,R)\big]^{\frac{-1}{2p(2p-1)}}\\
	&\le C_2(n,p)\max_{0\leq t\leq R}\frac{2t\sn_\lambda'(t)}{\sn_\lambda(t)}\|\rho_a(\nabla f)\|_{2p,f,a}(x,R)\\
	&\le C_2(n,p)\max_{0\leq t\leq R}\frac{2t\sn_\lambda'(t)}{\sn_\lambda(t)}\,\left(\|\rho_a(\nabla f)\|_{2p,f,a}+\|{\mathrm{Ric}_f^\lambda}_-\|^{\frac{1}{2}}_{p,f,a}\right).
	\end{align*}
	In the last second line we used H\"{o}lder's inequality. Combined with Proposition \ref{reduce}, we obtain that for any $R>0$, the following inequality holds:
	\begin{equation}
		\begin{aligned}
			&\quad\,\frac{d}{dR}\left(\frac{\vol_fB(x,R)}{v_a(n,\lambda,R)}\right)^\frac{1}{2p}\\
			&\leq \frac{C(n,p)Re^{2aR}\omega_\lambda(R)}{\left(\int_{0}^R\omega^a_\lambda(r)dr\right)^{1+\frac{1}{2p}}}\max_{0\leq t\leq R}\frac{t\sn_\lambda'(t)}{\sn_\lambda(t)}\\
			&\times \left(\|\rho_a(\nabla f)\|_{2p,f,a}(x,R)+\|{\mathrm{Ric}_f^\lambda}_-\|_{p,f,a}^{\frac{1}{2}}(x,R)\right).
		\end{aligned}\label{deri}
	\end{equation}
	Clearly, 
	\[
	\frac{C(n,p)Re^{2aR}\omega_\lambda(R)}{\left(\int_{0}^R\omega^a_\lambda(r)dr\right)^{1+\frac{1}{2p}}}\max_{0\leq t\leq R}\frac{t\sn_\lambda'(t)}{\sn_\lambda(t)}\sim C(n,p)\,R^{-\frac{n}{2p}}
	\]
	as $R\to 0$. Since $p>n/2$, the right-hand side of (\ref{deri}) is integrable. Hence, by integrating (\ref{deri}) from $r$ to $R$, we get the desired inequality
	\begin{equation*}
		\begin{aligned}
			&\left(\frac{\vol_f B(x,R)}{v_a(n,\lambda,R)}\right)^{\frac{1}{2p}}-\left(\frac{\vol_f B(x,r)}{v_a(n,\lambda,r)}\right)^{\frac{1}{2p}}\\
			&\leq
			C(n,p,a,\lambda,R)\left(\|\rho_a(\nabla f)\|_{2p,f,a}(x,R)+\|{\mathrm{Ric}_f^\lambda}_-\|^{\frac{1}{2}}_{p,f,a}(x,R)\right),
		\end{aligned}
	\end{equation*}
 where the constant  $C(n,p,a,\lambda,R)$ is given by
 \begin{equation}
C(n,p)\max_{0\leq t\leq R}\frac{t\sn_\lambda'(t)}{\sn_\lambda(t)}\,\int_{0}^{R}\frac{re^{2ar}\omega_\lambda(r)}{\left(\int_{0}^{r}\omega_\lambda^a(s)ds\right)^{1+\frac{1}{2p}}}dr. \label{coeffi}
 \end{equation}
 Moreover,  $C(n,p,a,\lambda,R)=O(R^{1-\frac{n}{2p}})$ as $R\to 0$.
\end{proof}

\begin{remark}
	By (\ref{coeffi}), one can   estimate the growth of $C(n,p,a,\lambda,R)$ as $R\to\infty$ as well.
\end{remark}

\section{An application}\label{sec5}

In this  section, we apply Theorem \ref{YeZ-theorem} to  prove  Corollary \ref{TZZZZ} for  the evolved metrics along the K\"ahler-Ricci flow (\ref{krf}). First we note that Zhang [\cite{Zhang}] proved that
\begin{align}\label{Zhang}|B(x, r, t)|_{g(t)} \leq c_1 r^n, ~\forall ~r\le {\rm diam}(M, g(t)),
\end{align}
where $c_1>0$ is a uniform constant independent of $t$.  By Perelman's non-collapsing result in [\cite{Pe}]  and the scalar curvature estimate [\cite{Pe2}] (also see [\cite{ST}]),   the uniform lower bound of volume is true,
\begin{align}\label{Pere}
	|B(x, r, t)|_{g(t)} \ge c_2 r^n, ~\forall ~r\le {\rm diam}(M, g(t)).
\end{align}
The diameter of $(M, g(t))$ is also uniformly bounded [\cite{Pe2}].

The potential function $f=f_t$ for  BE Ricci curvature in (\ref{potential}) associated to the metric $g=g(t)$  has been  constructed  in a recent work of Tian-Zhang-Zhang-Zhu-Zhu [\cite{TZZZZ}]  as follows.

Let $V=\mathrm{Ric}_-(g)\ge 0$. As in [\cite{TZZZZ}], let $\phi>0$ be the normalized first eigenfunction of
\begin{equation}\label{eigen}
	\Delta\phi+V\phi-\rho_0\phi=0,
\end{equation}
where $-\rho_0$ is the first eigenvalue of $-\Delta-V$, and set $f=\log\phi$. Then
\begin{equation}\label{eigenfcn}
	\Delta f+|\nabla f|^2+V=\rho_0,
\end{equation}
with $\rho_0\le B_0$ uniformly in $t$, which was proved in [\cite{TZZZZ}]. Moreover, we have 

\begin{lemma}\label{f}
	There exist positive constants $A$, $B_p\,(p\geq 1)$ and $C_\alpha\,(0<\alpha<1)$, only depending on the initial metric  $g_0$ and $p$ or $\alpha$  such that
	\begin{equation}\label{f1}
		|f(x)|\leq A,\quad \forall\, x\in M,
	\end{equation}
	\begin{equation}\label{f2}
		\|\nabla f\|_{L^p(M)}\leq B_p, \quad\forall\, p\geq 1,
	\end{equation}
	and
	\begin{equation}\label{f3}
		\|f\|_{C^{\alpha}(M)}\leq C_\alpha,\quad\forall\, 0<\alpha<1.
	\end{equation}
\end{lemma}

By a standard smooth approximation of $V$, we may assume that the solution  $f$ of (\ref{eigenfcn}) is smooth and all estimates  of $f$ in Lemma \ref {f} hold.

\begin{proof}[Proof of Lemma \ref{f}] (\ref{f1}) and (\ref{f2}) were proved in  [\cite{TZZZZ}, Proposition 2.2(a)] and [\cite{TZZZZ}, Proposition 2.3], respectively. So it remains to prove (\ref{f3}). 
	
	 Fix $q>n$. Let $B_j=B(x,2^{-j}R)$ and let $f_j=f_{B_j}$ be the average of $f$ over $B_j$. Then by Cauchy-Schwarz inequality,
	\begin{align*}
		|f_j-f_{j+1}|&\leq \fint_{B_{j+1}}|f-f_j|\leq \frac{1}{|B_{j+1}|^{1/2}}\left(\int_{B_{j}}|f-f_j|^2\right)^{\frac{1}{2}}.
	\end{align*}
	Since the following  uniform $L^2$-Poincar\'{e} inequality  holds along the K\"{a}hler-Ricci flow [\cite{TZq}],
	\[
	\left(\int_{B_{j}}|f-f_j|^2\right)^{\frac{1}{2}}\leq C\cdot 2^{-j}R\left(\int_{B_{j}}|\nabla f|^2\right)^{\frac{1}{2}},
	\]
by  H\"{o}lder's inequality with helps of  (\ref{Zhang}) and (\ref{Pere}),  we get
	\begin{align*}
		|f_j-f_{j+1}|&\leq \frac{C\cdot2^{-j}R}{|B_{j+1}|^{1/2}}\left(\int_{B_{j}}|\nabla f|^q\right)^{\frac{1}{q}}|B_j|^{\frac{1}{2}-\frac{1}{q}}\leq \tilde{C}(2^{-j}R)^{1-\frac{n}{q}}\|\nabla f\|_{L^q(M)}.
	\end{align*}
	Summing over $j$  it yields
	\begin{equation}\label{poin1}
	|f(x)-f_{B(x,R)}|\leq \tilde{C}_qR^{1-\frac{n}{q}}\|\nabla f\|_{L^q(M)},
	\end{equation}
	where $f_{B(x,R)}$ is the average of $f$ over $B(x,R)$.
	
	For any two points  $x,\,y\in M$,   let $B=B(x,2d)$ and $B_y=B(y,d)$, where $d=d(x,y)$. Then $B_y\subset B$.  Thus similar to (\ref{poin1}), we have
	\begin{align*}
		|f_B-f_{B_y}|\leq\frac{1}{|B_y|}\int_{B_y}|f-f_B| &\leq \frac{1}{|B_y|^{1/2}}\left(\int_{B}|f-f_B|^2\right)^{\frac{1}{2}}\\
		&\leq C_qd^{1-\frac{n}{q}}\|\nabla f\|_{L^q(M)}.
	\end{align*}
	Hence again  applying (\ref{poin1}) to $B$ and $B_y$, by  (\ref{f2}), we obtain 
	\[
	|f(x)-f(y)|\leq |f(x)-f_B|+|f_B-f_{B_y}|+|f(y)-f_{B_y}|\leq C_q\,d(x,y)^{1-\frac{n}{q}}.
	\]
 (\ref{f3}) is proved.
\end{proof}

For each $t>0$, define
\[
\tilde g(t)=e^{2f_t}g(t),
\quad F=(n-2)f_t,\quad |B(x,r,t)|_{\tilde{g},F}=\int_{B_{\tilde{g}}(x,r,t)}e^{-F}d\vol_{\tilde{g}}.
\]
Since $|f|$ is uniformly bounded, $g$ and $\tilde g$ are uniformly equivalent. 
\begin{lemma}\label{ball}
	For any $x\in M$ and $R>0$, let
	$k_R=\min_{B_g(x,R)}f$, $K_R=\max_{B_g(x,R)}f$.
	Then
	\[
	e^{-2K_R}|B(x,e^{k_R}R,t)|_{\tilde{g},F}\leq |B(x,R,t)|_{g}\leq e^{-2k_R}|B(x,e^{K_R}R,t)|_{\tilde{g},F}.
	\]
\end{lemma}
\begin{proof}
	By $d\tilde s=e^f ds$,  it is easy to see that	
	\[
	B_{\tilde{g}}(x,e^{k}R,t)\subset B_{g}(x,R,t)\subset B_{\tilde{g}}(x,e^{K}R,t).
	\]
	Moreover,
	\[
	d\vol_g=e^{-2f}e^{-F}d\vol_{\widetilde g}.
	\]
	Thus integrating over the above equality and using $k_R\le f\le K_R$ we prove the lemma.
\end{proof}

\begin{proof}[Proof of Corollary \ref{TZZZZ}]
	By  a formula of  Ricci tensor  for conformal metric  (cf. [\cite{Be}]), we have 
	\begin{align*}
		\mathrm{Ric}(\tilde{g})+(n-2) {\rm Hess}_g(f)=\mathrm{Ric}(g)-\Delta f\cdot g +(n-2)df\otimes df -(n-2)|\nabla f|^2 g.
	\end{align*}
	   Since
	   \[
\operatorname{Hess}_{\tilde{g}}(f)=\operatorname{Hess}_{g}(f)-2df\otimes df+|\nabla f|^2g,
\]
we get
	\begin{equation*}
	\operatorname{Ric}_F(\tilde{g})=\operatorname{Ric}(g)-\Delta f\cdot g-(n-2)df\otimes df.
	\end{equation*}
Since by (\ref{eigenfcn}), 
  $$\operatorname{Ric}(g)-\Delta f\cdot g\geq (-V-\Delta f)g\geq (-B_0+|\nabla f|^2)g, $$
	we derive
	 \begin{equation}\label{zhu}
		\mathrm{Ric}_{F}(\tilde{g})\geq -\left(A_0 + A_1|\nabla_{\tilde g}   f|_{\tilde g}^2\right)\tilde g,
	\end{equation}
	where $A_0=e^{2A}B_0$, $A_1=\max\{n-3,0\}$.  On the other hand, by Lemma \ref{f} together with the uniform bound of diameter,    for each $p\geq 1$, it holds
	\begin{align}\label{bound}
	\|{\operatorname{Ric}_{F}}_-(\tilde{g})\|^{\frac{1}{2}}_{p,F}(M)+\|\nabla_{\tilde{g}}F\|_{2p,F}(M)\leq D_p,
	\end{align}
where $D_p$ is independent of $t$.  Hence   applying Theorem \ref{YeZ-theorem} to $(M,\tilde{g},e^{-F}d\vol_{\tilde{g}})$ with $a=\lambda=0$ and $p>m$, 
  we obtain
	\begin{align}\label{p-volume}
			\left(\frac{ |B(x,R,t)|_{\tilde{g}, F}}{R^n}\right)^{\frac{1}{2p}}-\left(\frac{|B(x,r,t)|_{\tilde{g}, F}}{r^n}\right)^{\frac{1}{2p}}\leq C(n,p)\,R^{1-\frac{n}{2p}}\,D_p.
		\end{align}
By Lemma \ref{ball} and (\ref{Zhang}),  we see that
\[
\frac{|B(x,R,t)|_{\tilde{g},F}}{R^n}\leq e^{(n+2)A}\frac{|B(x,e^{A}R,t)|_g}{(e^{A}R)^n}\leq c_1\,e^{(n+2)A}.
\]
Applying the mean value theorem to $x^{\frac{1}{2p}}$, it follows that
	\begin{align*}
	&\quad\,\left(\frac{|B(x,R,t)|_{\tilde{g}, F}}{R^n}\right)^{\frac{1}{2p}}-\left(\frac{|B(x,r,t)|_{\tilde{g}, F}}{r^n}\right)^{\frac{1}{2p}}\\&\geq\frac{(c_1\,e^{(n+2)A})^{\frac{1}{2p}-1}}{2p}\left(\frac{|B(x,R,t)|_{\tilde{g}, F}}{R^n}-\frac{|B(x,r,t)|_{\tilde{g}, F}}{r^n}\right).
	\end{align*}
Consequently,  by (\ref{p-volume}) there exists a uniform constant $b_p$ depending only on $p$ and the initial metric $g_0$ such that
	\begin{align}\label{p-volume-2}
	\frac{|B(x,R,t)|_{\tilde{g},F}}{R^n}-\frac{|B(x,r,t)|_{\tilde{g},F}}{r^n}\leq b_pR^{1-\frac{m}{p}}.
	\end{align}
 Let
\[
k=\min_{B_g(x,R)}f,\quad K=\max_{B_g(x,R)}f.
\]
 Then  applying  Lemma \ref{ball} to $e^K R$ and $e^k r$  instead of $R$ and $r$ respectively,  
 we get from (\ref{p-volume-2}), 
\[
e^{2k-nK}\frac{|B(x,R,t)|_g}{R^n}-e^{2K-nk}\frac{|B(x,r,t)|_g}{r^n}\leq b_p ( e^KR)^{1-\frac{m}{p}}.
\]
Thus
\begin{align}
&e^{2k-nK}\bigg[\frac{|B(x,R,t)|_g}{R^n}-\frac{|B(x,r,t)|_g}{r^n}\bigg]\notag\\
&\le (e^{2K-nk}-e^{2k-nK})\frac{|B(x,r,t)|_g}{r^n}
+ b_p(e^KR)^{1-\frac{m}{p}}.\notag\\
&\le  (n+2)e^{(n+2)A}C_p R^{1-\frac{m}{p}}+b_p (e^KR)^{1-\frac{m}{p}}\notag\\
&\le b_p'R^{1-\frac{m}{p}},\notag
\end{align}
where $A$ and $C_p$ are two constants determined in (\ref{f1}) and (\ref{f3}), respectively.
In the last second inequality we used (\ref{f3}) for $\alpha=1-\frac{m}{p}$. 
Hence,  by setting $q=2p\in (n,+\infty)$,   we prove (\ref{TZZZZ-monotone}).
\end{proof}


\begin{thebibliography}{99}


\bibitem{BE}  D.  Bakry  and  M.  Emery,   Diffusions hypercontractives,  In S\'{e}minaire de probabilit\'{e}s, XIX, 1983/84,  Lecture Notes in Math.,  vol.  1123,  177-206,  Springer, Berlin, 1985.

\bibitem {Be} A. Besse,  {\it Einstein Manifolds}. Berlin, Heidelberg: Springer-Verlag, 2007.

\bibitem{CZ} H. Cao and D. Zhou, {\it  On  complete  gradient shrinking Ricci solitons}, J. Differential Geom.  85  (2010), no. 2, 175-186.

\bibitem{CQ}  H. Cao  and  Q. Chen,
\textit{On Bach-flat gradient shrinking Ricci solitons},
Duke Math. J. 162 (2013), 1149-1169.


\bibitem{CC} J. Cheeger and T. Colding, {\it On the structure of spaces with Ricci curvature bounded below. I}, J. Differential Geom.  46  (1997),  no. 3, 406-480.


\bibitem{CCT} J. Cheeger,  T. Colding and G.  Tian,  {\it On the singularities of spaces with bounded Ricci curvature},
Geom. Funct. Anal. 12 (2002), no. 5, 873-914.

\bibitem {Ha} R. Hamilton,
\textit{ Formation of singularities in the Ricci flow},  Surveys in Diff. Geom., 2 (1995), 7-136.

\bibitem{Qi} Z. Qian,  {\it   Estimates for weighted volumes and applications}.  Quart. J. Math., 48 (1997), no. 2, 235-242.

\bibitem{Pe} G. Perelman, {\it The entropy formula for the Ricci flow and its
geometric applications},  arXiv: math/0211159, 2002.

\bibitem {Pe2} G. Perelman,  unpublished, 2003.

\bibitem{PW} P. Petersen and  G. Wei, {\it Relative volume comparison with integral curvature bounds}, Geom. Funct. Anal. 7 (1997), no. 6, 1031-1045.

\bibitem{PS} P. Petersen and C. Sprouse, {\it Integral curvature bounds,  distance estimates and applications}. J. Differential Geom. 50 (1998), no. 2, 269-298.

\bibitem{PW2} P. Petersen and  G. Wei, {\it  Analysis and geometry on manifolds  with integral Ricci curvature bounds. II}, Trans. Amer. Math. Soc. 353 (2001), no. 2, 457-478.

\bibitem{ST}  N. Sesum  and  G. Tian,  \textit{Bounding scalar curvature and diameter along the K\"{a}hler-Ricci flow (after Perelman)}, J. Inst. Math. Jussieu, 7 (2008), no. 3,  575-587.

\bibitem{Ti} G. Tian, {\it  K\"ahler-Einstein metrics with positive scalar curvature}, Invent. Math., 130 (1997), 1-37.



 \bibitem{TZ}   G. Tian and  X.H. Zhu,  {\it A new holomorphic invariant and uniqueness of K\"ahler-Ricci
 solitons}, Comm. Math. Helv.  77  (2002), 297-325.

  \bibitem{TZZZ}  G.Tian, S.  Zhang,  Z. Zhang  and  X.H. Zhu, {\it Perelman's entropy and K\"ahler-Ricci flow on a Fano manifold}, Trans. Amer. Math. Soc.  365  (2013), no. 12, 6669-6695.

\bibitem{TZZZZ} G. Tian,  Q. Zhang, Z. Zhang, M. Zhu and X.H. Zhu,  {\it Laplace comparison on K\"ahler Ricci flow and convergence},  Peking Math. J., to appear, doi.org/10.1007/s42543-026-00136-3.

\bibitem{TZq} G. Tian and Q. Zhang, {\it Isoperimetric inequality under K\"ahler Ricci Flow}, Amer. J. Math. 136 (2014), no. 5, 1155-1173.


\bibitem{TZhangZ1} G. Tian  and Z. Zhang,
{\it Regularity of  K\"ahler Ricci solitons},  Intern.  Math.  Res.  Notices,  2012, 957-985.


\bibitem{TZhangZ} G. Tian and Z.  Zhang,
{\it Regularity of  K\"ahler Ricci flows on Fano manifolds}, Acta Math. 216 (2016), no. 1, 127-176.
\bibitem{TZZ}  G. Tian,   L. Zhang  and  X. H. Zhu,  {\it K\"ahler-Ricci  flow for deformed  complex structures}, Trans. Amer. Math. Soc. 376 (2023), 1999-2046.

\bibitem{WZ} F. Wang and X.H.  Zhu,  {\it Structure of spaces with Bakry-Emery Ricci curvature bounded below},   Journal f\"ur die reine und angewandte Mathematik (Crelles Journal), vol. 2019, no. 757, 2019, pp. 1-50.

\bibitem{WW}  G. Wei  and  W. Wylie,  Comparison geometry for Bakry-\'Emery Ricci curvature,  J. Differential Geom. 83 (2009), 337-405.

\bibitem{Wu} J. Wu,  {\it  Comparison geometry for integral  Bakry-\'Emery Ricci tensor bounds}, Jour. of Geom. Analysis, vol. 29 (2019), no. 1, 828-867.

\bibitem{Zhang} Q. Zhang, {\it Bounds on volume growth of geodesic balls under Ricci flow},
Math. Res. Lett. 19 (2012), no. 1, 245-253.

\end{thebibliography}
\end{document}